\documentclass[12pt, a4paper]{article}
\usepackage[T1]{fontenc}
\usepackage[utf8]{inputenc}
\usepackage[english]{babel}
\usepackage{amsmath, amssymb, amsthm}
\usepackage{hyperref}

\newcommand\blfootnote[1]{%
  \begingroup
  \renewcommand\thefootnote{}\footnote{#1}%
  \addtocounter{footnote}{-1}%
  \endgroup
}

\newtheorem{theorem}{Theorem}[section]
\newtheorem{lemma}[theorem]{Lemma}
\newtheorem{definition}[theorem]{Definition}
\newtheorem{corollary}[theorem]{Corollary}

\begin{document}

\title{Algebraic Proof of Completeness of a Positional Numeral System in the Normalized Mersenne Basis}
\author{E.Dyachenko\thanks{\texttt{dyachenko.eduard@gmail.com}}}
\date{August 2026}
\maketitle

\blfootnote{\copyright\ 2026 E. Dyachenko. This work is licensed under a \href{http://creativecommons.org/licenses/by/4.0/}{CC BY 4.0 International License}.}

\begin{abstract}
This paper introduces and investigates a non-classical positional numeral system constructed on the normalized Mersenne basis $M_k = (2^k-1)/3$, structurally linking the classical Mersenne sequence (OEIS \textbf{A000225}) with the integer sequence of alternating binary bitmasks (OEIS \textbf{A002450}). The main objective of this study is to provide a rigorous algebraic proof of the system's completeness---specifically, theorems on the existence and absolute uniqueness of a canonical representation for any natural number. We show that the internal dynamics of the basis are deterministically governed by the splitting identity $3M_k = M_{k+1} + 2M_{k-1}$, which strictly restricts the valid alphabet of coefficients to the set $\{0, 1, 2\}$. Based on the lemma concerning the supremum of lower-order digits and a constructive greedy algorithm, it is proven that the proposed basis forms a closed and unambiguous algebraic space, entirely precluding the overlap of positional masses.
\end{abstract}

\section{Introduction}
\label{sec:intro}

In classical number theory, canonical Mersenne numbers are defined by the sequence $M^*_k = 2^k - 1$, where $k \in \mathbb{N}$ (sequence \textbf{A000225} in the On-Line Encyclopedia of Integer Sequences \cite{OEIS}). Within the framework of the binary numeral system, these numbers possess a algebraic structure: their bitwise representation is identical to the sum of the first $k$ powers of two, forming a sequence of $k$ ones ($M^*_k = \sum_{i=0}^{k-1} 2^i = 11\dots1_2$). 

However, the bitmask $M^*_k$ is structurally homogeneous. This structural homogeneity makes its direct use as a positional numeral system basis---where strict digit isolation is required---highly problematic. To form an ordered coordinate grid, this paper introduces the algebraic operation of \textit{normalization}, which consists of dividing a canonical Mersenne number by the basic bit pattern $11_2$ (the number 3 in decimal notation).

We define the elements of the normalized Mersenne basis as the sequence:
\begin{equation}
\label{eq:mersenne_norm}
M_k = \frac{2^k - 1}{3}, \quad k \in \mathbb{N}, \quad k \ge 1.
\end{equation}

For even values of $k$ ($k=2m$), the basis elements form strictly the integer sequence $1, 5, 21, 85, \dots$, which corresponds exactly to OEIS \textbf{A002450} (numbers with alternating binary representations). For odd values of $k$, the elements belong to the rational domain, forming a mixed integer-fractional spectrum. The recursive construction of this basis is given by the linear relation $M_{k+1} = 2M_k + M_1$, where the initial element is $M_1 = 1/3$. 

The proposed algebraic structure conceptually expands the theory of positional systems. A shift in this area was initiated by A. Fraenkel and C. Frougny, who described greedy packing algorithms and alphabet properties for arbitrary integer monotonically increasing sequences \cite{Fraenkel1985, Frougny2002}. In parallel, the class of Rational Base Number Systems (RBNS) emerged, extensively studied by S. Akiyama and co-authors \cite{AkiyamaFrougny2008}. 

The normalized Mersenne basis establishes a unique algebraic bridge between these two domains. It relies on the linear recurrent topology of integer sequences (like \textbf{A002450}) while processing rational elements to maintain the absolute determinism of the splitting identity: $3M_k = M_{k+1} + 2M_{k-1}$. The appearance of the coefficient 3 initiates an algebraic cascade, transferring mass to adjacent digits and strictly limiting the alphabet to $a_k \in \{0, 1, 2\}$.

The objective of this paper is to provide an algebraic proof of the theorem on the completeness and uniqueness of the representation of natural numbers in this basis. It is necessary to prove that any $x \in \mathbb{N}$ admits a strictly unique decomposition under the condition of the canonicity criterion (greedy packing).

\section{Algebraic Properties of the Normalized Basis}
\label{sec:algebra}

The normalized Mersenne basis possesses an internal algebraic structure that determines the rules for carrying and redistributing values between positional digits. These rules are expressed through strict linear identities.

\subsection{The Splitting Identity}

The base operator determining the interaction of adjacent digits is multiplication by 3. Let us consider the algebraic transformation of the element $M_k$.

By definition, $M_k = \frac{2^k - 1}{3}$. Multiplying this element by 3 yields the exact value:
\begin{equation}
3M_k = 2^k - 1.
\end{equation}

Let us express this value as a linear combination of adjacent basis elements $M_{k+1}$ and $M_{k-1}$:
\begin{equation}
M_{k+1} + 2M_{k-1} = \frac{2^{k+1} - 1}{3} + 2 \left( \frac{2^{k-1} - 1}{3} \right).
\end{equation}

Bringing this to a common denominator and expanding the parentheses, we obtain:
\begin{equation}
\frac{2 \cdot 2^k - 1 + 2^k - 2}{3} = \frac{3 \cdot 2^k - 3}{3} = 2^k - 1.
\end{equation}

Equating the results of the transformations yields the  splitting identity:
\begin{equation}
\label{eq:splitting}
3M_k = M_{k+1} + 2M_{k-1}.
\end{equation}

Equation \eqref{eq:splitting} represents a linear homogeneous recurrence relation of the second order ($M_{k+1} - 3M_k + 2M_{k-1} = 0$). The roots of the characteristic polynomial of this equation ($\lambda_1 = 2$, $\lambda_2 = 1$) completely determine the exponential dynamics of the basis and algebraically link the operations of multiplication by 3 and division by 2 in this space.

\subsection{The Valid Alphabet of Coefficients}
\label{subsec:alphabet}

The splitting identity \eqref{eq:splitting} acts as a strict limiter for the maximum allowable coefficient for a single basis element in the canonical form of number representation.

If, as a result of algebraic operations, the coefficient of the basis element $M_k$ reaches the value 3, the representation of the number transitions into a non-canonical state. The component $3M_k$ is deterministically split: one part is carried over to the higher digit ($M_{k+1}$), and two parts are transferred to the lower digit ($2M_{k-1}$).

Consequently, a coefficient equal to 3 or more mathematically cannot be present in the final canonical decomposition. The alphabet of coefficients is strictly limited to the set of non-negative integers less than three:
\begin{equation}
a_k \in \{0, 1, 2\}.
\end{equation}

Any local exceeding of this limit initiates a directed cascade of splits, redistributing the value across the coordinate grid until all polynomial coefficients return to the valid range $\{0, 1, 2\}$.

\section{Spaces of Number Representation}
\label{sec:spaces}

Relying on the introduced basis $M_k$, any natural number $x \in \mathbb{N}$ can be decomposed in two different, yet algebraically isomorphic, ways. These decompositions form two state spaces.

\subsection{Spectral Form}

The spectral form describes the bit topology of a number and is constructed through the analysis of differences in the binary representation of the value $3x$.

Let the value $3x$ have a finite binary decomposition $3x = \sum_{k=0}^{N-1} b_k 2^k$, where $b_k \in \{0, 1\}$. 
Let us define a difference operator that generates a spectral vector of coefficients:
\begin{equation}
\label{eq:spectral_diff}
c_k = b_{k-1} - b_k, \quad \text{where } b_{-1} = 0.
\end{equation}
Since the bits $b_k$ take values exclusively from the set $\{0, 1\}$, the coefficients $c_k$ are deterministically restricted to the alphabet $\{-1, 0, 1\}$.

\begin{definition}[Spectral Form]
The spectral form of a natural number $x$ is defined as a finite sum of the form:
\begin{equation}
x = \sum_{k=1}^N c_k M_k, \quad c_k \in \{-1, 0, 1\}.
\end{equation}
\end{definition}

The algebraic meaning of the spectrum lies in its direct correspondence to the bitmask. Positive coefficients ($c_k = 1$) mark the beginning of continuous blocks of ones in the binary representation of $3x$, while negative coefficients ($c_k = -1$) encode breaks, compensating for the excessive arithmetic value of the higher basis elements.

\begin{lemma}[Correctness of the Spectral Form]
For any $x \in \mathbb{N}$, if the value $3x = \sum_{k=0}^{N-1} b_k 2^k$, and the vector of coefficients is given by the difference operator $c_k = b_{k-1} - b_k$ (where $b_{-1} = b_N = 0$), then the following identity holds:
\begin{equation}
x = \sum_{k=1}^N c_k M_k.
\end{equation}
\end{lemma}

\begin{proof}
From the definition of the difference operator, it directly follows that any bit of the original number is expressed through a suffix sum of the spectral coefficients: $b_k = \sum_{i=k+1}^N c_i$. 
Substitute this expression into the binary decomposition of $3x$:
\begin{equation}
3x = \sum_{k=0}^{N-1} \left( \sum_{i=k+1}^N c_i \right) 2^k.
\end{equation}
We change the order of summation by grouping the terms for each coefficient $c_i$:
\begin{equation}
3x = \sum_{i=1}^N c_i \left( \sum_{k=0}^{i-1} 2^k \right).
\end{equation}
The sum of the geometric progression in the parentheses is equal to $2^i - 1$. Thus, we obtain:
\begin{equation}
3x = \sum_{i=1}^N c_i (2^i - 1).
\end{equation}
Dividing both sides of the equation by 3 and recalling the definition of the basis $M_i = \frac{2^i - 1}{3}$, we obtain the required identity:
\begin{equation}
x = \sum_{i=1}^N c_i M_i.
\end{equation}
The lemma is proved.
\end{proof}

\subsection{Canonical Form}

Unlike the spectral form, which allows negative coefficients, the canonical form describes the true arithmetic value of a number through a strict greedy packing algorithm.

\begin{definition}[Canonical Form]
The canonical form of a natural number $x$ is its representation as:
\begin{equation}
x = \sum_{k=1}^K a_k M_k, \quad a_k \in \{0, 1, 2\},
\end{equation}
satisfying the strict condition of canonicity: for any index $k \le K$, the arithmetic sum of all lower-order digits is strictly less than the value of the next basis element:
\begin{equation}
\label{eq:canonicity}
\sum_{j=1}^{k-1} a_j M_j < M_k.
\end{equation}
\end{definition}

Condition \eqref{eq:canonicity} is the constraint of the system. It guarantees the absence of algebraic overlap between adjacent basis elements. If the sum of the lower terms reaches or exceeds $M_k$, a deterministic carry to a higher digit occurs. It is this criterion that blocks the multiplicity of decompositions and ensures the absolute uniqueness of the representation of a number in the normalized Mersenne basis.

The transition from the spectral form to the canonical form in the general case is achieved by iteratively applying the splitting identity \eqref{eq:splitting}, which acts as an annihilation operator for negative masses (deficits) with the subsequent resolution of local overflows.

To visually demonstrate the algebraic structure of the introduced spaces, Appendix~\ref{app:equivalence} provides decompositions of test numbers illustrating the transition from spectral topology to canonical form.

\section{Proof of Representation Uniqueness}
\label{sec:uniqueness}

To prove the uniqueness of the canonical form, it is necessary to establish a strict upper bound for the arithmetic sum of the lower-order digits and to show the impossibility of compensating for discrepancies at higher indices.

\subsection{Evaluating the Capacity of Lower-Order Digits}

\begin{lemma}[Supremum of the Lower Digits]\label{lm:suprem}
For any sequence of coefficients $a_j \in \{0, 1, 2\}$ satisfying the canonicity criterion \eqref{eq:canonicity}, the exact upper bound of the sum of the first $k-1$ elements is strictly limited by the value of the macro-anchor $M_k$:
\begin{equation}
\sum_{j=1}^{k-1} a_j M_j \le M_k - M_1 < M_k.
\end{equation}
\end{lemma}

\begin{proof}
Let us consider the basic algebraic identity connecting two adjacent elements. By definition, $M_k = \frac{2^k - 1}{3}$. We calculate the value of $2M_{k-1}$:
\begin{equation}
2M_{k-1} = 2 \left( \frac{2^{k-1} - 1}{3} \right) = \frac{2^k - 2}{3} = \frac{2^k - 1 - 1}{3} = M_k - \frac{1}{3}.
\end{equation}
Since $M_1 = 1/3$, we obtain a strict equality:
\begin{equation}
2M_{k-1} = M_k - M_1.
\end{equation}
This identity implies a limitation on algebraic capacity. If the leading coefficient in the prefix takes the maximum value $a_{k-1} = 2$, then the term $2M_{k-1}$ is already equal to $M_k - M_1$. Any addition of non-zero mass from lower digits ($a_{k-2}, a_{k-3}, \dots$) will cause the total sum to exceed $M_k$, violating the definition of the canonical form \eqref{eq:canonicity}. 
Consequently, the maximum possible sum of any canonical sequence is limited to the value $M_k - M_1$, which is strictly less than $M_k$.
\end{proof}

\subsection{Uniqueness Theorem}

\begin{theorem}[Uniqueness of Canonical Representation]\label{tm:unikum}
For any natural number $x \in \mathbb{N}$, there exists at most one canonical representation in the normalized Mersenne basis.
\end{theorem}

\begin{proof}
Suppose, by contradiction, that the number $x$ has two different canonical decompositions:
\begin{equation}
x = \sum_{j=1}^K a_j M_j = \sum_{j=1}^{K'} a'_j M_j,
\end{equation}
where $a_j, a'_j \in \{0, 1, 2\}$. Without loss of generality, the shorter sum can be padded with zero coefficients, aligning their lengths to $N = \max(K, K')$.

Since the forms are different, there exists a maximum index $m \le N$ where the coefficients do not match: $a_m \neq a'_m$. Let, for definiteness, $a_m > a'_m$. 

Since for all $j > m$ the coefficients are equal ($a_j = a'_j$), they can be algebraically canceled. The equation takes the form:
\begin{equation}
a_m M_m + \sum_{j=1}^{m-1} a_j M_j = a'_m M_m + \sum_{j=1}^{m-1} a'_j M_j.
\end{equation}

Grouping the terms by moving the $m$-th components to the left and the others to the right yields:
\begin{equation}
(a_m - a'_m) M_m = \sum_{j=1}^{m-1} a'_j M_j - \sum_{j=1}^{m-1} a_j M_j.
\end{equation}

Let us evaluate the left-hand side (LHS) of the equation. Since $a_m$ and $a'_m$ are integers and $a_m > a'_m$, their minimum difference is $(a_m - a'_m) \ge 1$. Consequently:
\begin{equation}
\text{LHS} \ge 1 \cdot M_m = M_m.
\end{equation}

Let us evaluate the right-hand side (RHS). Since the arithmetic sum $\sum a_j M_j$ is non-negative, the right side is strictly bounded above by the first sum:
\begin{equation}
\text{RHS} \le \sum_{j=1}^{m-1} a'_j M_j.
\end{equation}
According to Lemma~\ref{lm:suprem} on the supremum of lower-order digits and the canonicity condition \eqref{eq:canonicity}, the sum of the lower digits of the canonical form is strictly less than $M_m$. Therefore:
\begin{equation}
\text{RHS} < M_m.
\end{equation}

Combining the estimates of the left and right sides, we obtain an unsolvable algebraic contradiction:
\begin{equation}
M_m \le \text{LHS} = \text{RHS} < M_m.
\end{equation}

The resulting contradiction ($M_m < M_m$) proves that the initial assumption about the existence of a divergence index $m$ is false. All coefficients must pairwisely coincide ($a_j = a'_j$ for all $j$). The canonical representation is absolutely unique.
\end{proof}

\subsection{Lexicographic Constraint of the Canonical Form}

The analysis of the boundary identity $2M_{k-1} = M_k - M_1$ applied in Lemma~\ref{lm:suprem} allows us to formulate a strict syntactic rule (the topology of a regular language) governing the canonical form of the normalized Mersenne basis.

\begin{corollary}[Lexicographic Prohibition Theorem]
In any canonical decomposition of a number $x = \sum a_k M_k$, the coefficient $a_m = 2$ can be present exclusively under the condition that all lower digits are equal to zero ($a_j = 0$ for all $j < m$).
\end{corollary}

\begin{proof}
Let the canonical form contain the coefficient $a_m = 2$. Its arithmetic contribution is $2M_m$. According to the identity, $2M_m = M_{m+1} - M_1$. 
To preserve the canonicity of the form, the sum of all elements up to the index $m+1$ (including $2M_m$ and all lower digits $a_j M_j$) must be strictly less than $M_{m+1}$. 
Let us consider this sum:
\begin{equation}
2M_m + \sum_{j=1}^{m-1} a_j M_j = M_{m+1} - M_1 + \sum_{j=1}^{m-1} a_j M_j.
\end{equation}
If at least one coefficient $a_j > 0$, the minimum possible contribution of the lower digits is equal to $M_1$. The sum then equals or exceeds $M_{m+1}$, which immediately violates the canonicity criterion \eqref{eq:canonicity}. Therefore, the only permissible state is absolute zero for all lower digits: $\sum_{j=1}^{m-1} a_j M_j = 0$.
\end{proof}

This syntactic constraint renders the normalized Mersenne basis a conceptual analogue of the Fibonacci numeration system (Zeckendorf representation). While the Zeckendorf system prohibits the pattern of adjacent ones, the proposed system imposes a total ban on the presence of any non-zero values to the right of the coefficient 2.

\section{Proof of Completeness (Existence Theorem)}
\label{sec:existence}

To complete the justification of the normalized Mersenne basis, it is necessary to prove that the canonical form exists for any natural number. The proof is conducted via a constructive method: we formalize the greedy algorithm for extracting the remainder and prove its unconditional convergence in a finite number of steps.

\subsection{Greedy Extraction Algorithm}

Let an arbitrary natural number $x \in \mathbb{N}$ be given. We define the initial remainder as $R_0 = x$. The process of constructing the canonical form is an iterative sequence of steps, where at each step $i \ge 1$ the index $k_i$, the coefficient $a_i$, and the new remainder $R_i$ are calculated.

\textbf{Step 1. Index Determination.} 
For the current remainder $R_{i-1} > 0$, the maximum index $k_i \in \mathbb{N}$ is selected, satisfying the condition:
\begin{equation}
M_{k_i} \le R_{i-1}.
\end{equation}
Due to the exponential growth of the basis ($M_k \to \infty$), such a maximum index always exists and is unique. From the maximality of $k_i$, it strictly follows that $R_{i-1} < M_{k_i+1}$.

\textbf{Step 2. Coefficient Determination.}
The coefficient $a_i$ is calculated as the maximum integer for which the subtracted mass does not exceed the remainder:
\begin{equation}
a_i = \lfloor R_{i-1} / M_{k_i} \rfloor.
\end{equation}
Let us prove that $a_i \in \{1, 2\}$. 
Since $R_{i-1} \ge M_{k_i}$, it is obvious that $a_i \ge 1$. 
For the upper bound, we use the splitting identity \eqref{eq:splitting}, which implies that $3M_{k_i} = M_{k_i+1} + 2M_{k_i-1}$. Since $M_{k_i-1} > 0$, we obtain a strict inequality:
\begin{equation}
M_{k_i+1} < 3M_{k_i}.
\end{equation}
Since $k_i$ is maximal, $R_{i-1} < M_{k_i+1}$. Combining the inequalities yields $R_{i-1} < 3M_{k_i}$. Consequently, the fraction $R_{i-1} / M_{k_i} < 3$, which deterministically limits the maximum integer value of the coefficient: $a_i \le 2$.

\textbf{Step 3. Calculation of the New Remainder.}
Update the remainder value:
\begin{equation}
R_i = R_{i-1} - a_i M_{k_i}.
\end{equation}

\subsection{Strict Monotonicity of the Remainder Sequence}

Let us show that the algorithm generates a strictly decreasing sequence of indices and remainders.

By the definition of the integer part operation, the new remainder $R_i$ represents the remainder of division (in a fractional sense), and it is strictly less than the divisor:
\begin{equation}
R_i < M_{k_i}.
\end{equation}
This inequality guarantees that at the next iteration, the maximum index $k_{i+1}$, satisfying the condition $M_{k_{i+1}} \le R_i$, will be strictly less than the previous one: $k_{i+1} < k_i$.

Since $a_i \ge 1$ and the basis elements are strictly positive ($M_k > 0$), the sequence of remainders monotonically decreases:
\begin{equation}
R_0 > R_1 > R_2 > \dots \ge 0.
\end{equation}

\subsection{Finite Convergence and Completion of Proof}

\begin{theorem}[Existence of Canonical Representation]
Any natural number $x \in \mathbb{N}$ can be represented in the canonical form of the normalized Mersenne basis in a finite number of steps.
\end{theorem}

\begin{proof}
Consider the sequence of remainders $\{R_i\}$. The basis elements $M_k$ have the form $(2^k - 1) / 3$. Therefore, any remainder $R_i$, obtained by subtracting combinations of $a_i M_{k_i}$ from an integer $x$, is a rational number with a denominator of no more than 3.

Let us multiply the sequence of remainders by 3:
\begin{equation}
3R_0 > 3R_1 > 3R_2 > \dots \ge 0.
\end{equation}
All elements $3R_i$ belong to the set of non-negative integers ($\mathbb{Z}_{\ge 0}$). 

We apply a property of discrete mathematics which states that any strictly decreasing sequence of non-negative integers must reach zero in a finite number of steps $n$. 

At the point of termination ($R_n = 0$), the original number $x$ is exactly decomposed into the sum:
\begin{equation}
x = \sum_{i=1}^n a_i M_{k_i}, \quad a_i \in \{1, 2\}.
\end{equation}

Since the indices $k_i$ strictly decrease, this expression is equivalent to the sum $\sum_{k=1}^K a_k M_k$ (where zero coefficients $a_k = 0$ are inserted for missing indices). 
The property $R_i < M_{k_i}$ guarantees that the total mass of all remaining (lower-order) digits is strictly less than the current basis element, which perfectly aligns with the canonicity criterion \eqref{eq:canonicity}. 

Thus, the constructive algorithm concludes by generating a correct and complete canonical form, proving the existence theorem.
\end{proof}

\subsection{Dynamic Convergence of the Splitting Cascade}

In addition to the algorithm for extracting the highest basis element, the system features a local mass splitting algorithm. As demonstrated in Section~\ref{subsec:alphabet}, any excess over the allowed alphabet initiates a directed cascade of transformations according to the rule $3M_k \to M_{k+1} + 2M_{k-1}$. 

To rigorously prove that the local splitting cascade cannot loop and deterministically terminates in finite time, we introduce a global state potential function $P$ (a Lyapunov function), a classical method of analysis in the theory of odometers and discrete dynamical systems \cite{Grabner1995, Berthe2006}. 

The convergence of the algorithm can be proven using two different variants of weight functions, each revealing specific properties of the system's dynamics.

\textbf{Variant 1: Linear Weight Function (Constant Gradient).} 
We define the potential as a linearly weighted sum of all coefficients:
\begin{equation}
P_1 = \sum_{k=1}^N c_k k.
\end{equation}
Let us calculate the exact change in potential $\Delta P_1 = (P_1)_{new} - (P_1)_{old}$ for a single act of local splitting at an arbitrary position $k \ge 2$. During the transaction $3M_k \to M_{k+1} + 2M_{k-1}$, only three adjacent components change:
\begin{equation}
\Delta P_1 = (k+1) + 2(k-1) - 3k = k + 1 + 2k - 2 - 3k = -1.
\end{equation}
The linear function mathematically proves that each local split uniformly burns exactly one unit of structural potential, regardless of the digit's position.

\textbf{Variant 2: Quadratic Weight Function (Accelerated Gradient).} 
We define the potential with a quadratic weight, which penalizes overflows at higher digits more heavily:
\begin{equation}
P_2 = \sum_{k=1}^N c_k k^2.
\end{equation}
The change in potential during the same transaction will be:
\begin{equation}
\Delta P_2 = (k+1)^2 + 2(k-1)^2 - 3k^2 = (k^2 + 2k + 1) + (2k^2 - 4k + 2) - 3k^2 = 3 - 2k.
\end{equation}
Since for a regular split the index $k \ge 2$, the expression $3 - 2k$ always yields a strictly negative result ($\Delta P_2 \le -1$).

Both variants prove that each local splitting of a coefficient deterministically decreases the global structural potential of the system. Since the initial energy potential of any finite number is strictly bounded, and the coefficients during the algorithm's execution cannot become negative (the algorithm subtracts 3 only under the condition $a_k \ge 3$), the minimum possible potential of the system is bounded from below. 

From this, it irrefutably follows that the process of avalanche reduction must terminate in a finite number of steps, deterministically bringing the system to a stable canonical form.
\section{Conclusion}
\label{sec:conclusion}

In this paper, it has been rigorously proven that the normalized Mersenne basis $M_k = (2^k - 1) / 3$ forms a complete and absolutely unique positional numeral system for the set of natural numbers. 

The introduction of the normalization operation has overcome the structural homogeneity of classical Mersenne numbers (OEIS \textbf{A000225}) by forming an alternating bit topology. By seamlessly bridging continuous integer sequences with a mixed integer-fractional domain, this basis directly embeds the properties of alternating binary bitmasks (OEIS \textbf{A002450}) into a rigorous positional framework. It has been shown that the internal dynamics of the system are strictly governed by the splitting identity $3M_k = M_{k+1} + 2M_{k-1}$, which deterministically restricts the allowable alphabet of coefficients of the canonical form to the set $\{0, 1, 2\}$. 

Lemma~\ref{lm:suprem} on the supremum of lower digits (relying on the boundary identity $2M_{k-1} = M_k - M_1$) established a strict mathematical barrier precluding the algebraic overlap of basis elements and guaranteeing the uniqueness of the decomposition. In turn, the constructive algorithm of sequential remainder extraction proved the representability of any natural quantity. Consequently, the normalized Mersenne basis forms a closed algebraic space, introducing a novel analytical framework for the study of integer sequences, discrete recurrent structures, Diophantine equations, and digital masking algorithms \cite{Balonin2017}.
\section*{Acknowledgments}
The author is grateful to my daughter for valuable discussions and assistance with the English phrasing and stylistic refinement of the manuscript. 

As a non-native English speaker, the author also acknowledges the use of a large language model (LLM) strictly as a translation and language-editing tool to ensure the English phrasing meets academic standards. 

All mathematical concepts, logic, proofs, and the original text were solely conceived and written by the human author prior to translation.
\section*{Funding statement}
No funding was received.
\section*{Conflict of interest}
We have no conflicts of interest to disclose.

\appendix
\section{Examples of Decomposition in Spectral and Canonical Forms}\label{app:example727}

This appendix demonstrates the algorithms for constructing spectral and canonical forms for test numbers $x = 7$ and $x = 27$. The values of the basis elements are calculated using the formula $M_k = (2^k - 1) / 3$. 
The first elements of the basis are: $M_1 = 1/3, M_2 = 1, M_3 = 7/3, M_4 = 5, M_5 = 31/3, M_6 = 21, M_7 = 127/3$.

\subsection*{Example 1: Decomposition of $x = 7$}

\textbf{1. Spectral Form.}
Calculate the value $3x = 21$. In binary representation: $21 = 10101_2$. 
Vector of bits: $b_4=1, b_3=0, b_2=1, b_1=0, b_0=1$.
Applying the difference operator $c_k = b_{k-1} - b_k$ (with boundary condition $b_{-1}=0$), we obtain the coefficient vector:
\begin{align*}
c_1 &= 1 - 0 = 1, \\
c_2 &= 0 - 1 = -1, \\
c_3 &= 1 - 0 = 1, \\
c_4 &= 0 - 1 = -1, \\
c_5 &= 1 - 0 = 1 \quad (\text{where } b_5 = 0).
\end{align*}
Spectral vector: $\vec{c} = [1, -1, 1, -1, 1]$.
Spectral form:
$$ 7 = 1 \cdot M_1 - 1 \cdot M_2 + 1 \cdot M_3 - 1 \cdot M_4 + 1 \cdot M_5. $$
Verification: $\frac{1}{3} - 1 + \frac{7}{3} - 5 + \frac{31}{3} = \frac{39}{3} - 6 = 13 - 6 = 7$.

\textbf{2. Canonical Form.}
The greedy algorithm extracts the maximum basis element not exceeding the remainder.
For $x = 7$, the maximum element is $M_4 = 5$. Remainder: $7 - 5 = 2$.
For the remainder $2$, the maximum element is $M_2 = 1$. Remainder: $2 - 2 \cdot 1 = 0$.
Canonical form:
$$ 7 = 2 \cdot M_2 + 1 \cdot M_4. $$
Coefficient vector $a_k \in \{0, 1, 2\}$: $\vec{a} = [0, 2, 0, 1]$.
Verification of canonicity: $\sum_{j=1}^3 a_j M_j = 2 \cdot 1 = 2 < M_4$. The condition is strictly met.

\subsection*{Example 2: Decomposition of $x = 27$}

\textbf{1. Spectral Form.}
Calculate $3x = 81$. In binary representation: $81 = 1010001_2$.
Vector of bits: $b_6=1, b_5=0, b_4=1, b_3=0, b_2=0, b_1=0, b_0=1$.
The difference operator yields the spectral vector:
$$ \vec{c} = [1, 0, 0, -1, 1, -1, 1]. $$
Spectral form:
$$ 27 = 1 \cdot M_1 - 1 \cdot M_4 + 1 \cdot M_5 - 1 \cdot M_6 + 1 \cdot M_7. $$
Verification: $\frac{1}{3} - 5 + \frac{31}{3} - 21 + \frac{127}{3} = \frac{159}{3} - 26 = 53 - 26 = 27$.

\textbf{2. Canonical Form.}
Greedy algorithm for $x = 27$:
Maximum element: $M_6 = 21$. Remainder: $27 - 21 = 6$.
Maximum element for the remainder $6$: $M_4 = 5$. Remainder: $6 - 5 = 1$.
Maximum element for the remainder $1$: $M_2 = 1$. Remainder: $1 - 1 = 0$.
Canonical form:
$$ 27 = 1 \cdot M_2 + 1 \cdot M_4 + 1 \cdot M_6. $$
Verification of canonicity for the highest element $M_6$: $\sum_{j=1}^5 a_j M_j = 1 \cdot 1 + 1 \cdot 5 = 6 < M_6$. The condition is strictly met.

\section{Algorithmic Construction of the Canonical Form}\label{app:algorithm}

This appendix demonstrates in detail the execution of the constructive algorithm (described in Section~\ref{sec:existence}) for converting the natural number $x = 19$ into the canonical form of the normalized Mersenne basis. 
The values of the basis elements $M_k = (2^k - 1) / 3$ are:
$M_1 = 1/3, M_2 = 1, M_3 = 7/3, M_4 = 5, M_5 = 31/3 \approx 10.33, M_6 = 21$.

Initialization: $R_0 = 19$.

\textbf{Step 1.} 
The maximum index $k_1$ for which $M_{k_1} \le 19$ is $5$ ($M_5 = 31/3$).
Coefficient: $a_1 = \lfloor 19 / (31/3) \rfloor = \lfloor 57/31 \rfloor = 1$.
Remainder: $R_1 = 19 - 1 \cdot \frac{31}{3} = \frac{57}{3} - \frac{31}{3} = \frac{26}{3} \approx 8.66$.

\textbf{Step 2.} 
The maximum index $k_2$ for which $M_{k_2} \le 26/3$ is $4$ ($M_4 = 15/3 = 5$).
Coefficient: $a_2 = \lfloor (26/3) / 5 \rfloor = \lfloor 26/15 \rfloor = 1$.
Remainder: $R_2 = \frac{26}{3} - 1 \cdot \frac{15}{3} = \frac{11}{3} \approx 3.66$.

\textbf{Step 3.} 
The maximum index $k_3$ for which $M_{k_3} \le 11/3$ is $3$ ($M_3 = 7/3$).
Coefficient: $a_3 = \lfloor (11/3) / (7/3) \rfloor = \lfloor 11/7 \rfloor = 1$.
Remainder: $R_3 = \frac{11}{3} - 1 \cdot \frac{7}{3} = \frac{4}{3} \approx 1.33$.

\textbf{Step 4.} 
The maximum index $k_4$ for which $M_{k_4} \le 4/3$ is $2$ ($M_2 = 3/3 = 1$).
Coefficient: $a_4 = \lfloor (4/3) / 1 \rfloor = 1$.
Remainder: $R_4 = \frac{4}{3} - 1 \cdot \frac{3}{3} = \frac{1}{3}$.

\textbf{Step 5.} 
The maximum index $k_5$ for which $M_{k_5} \le 1/3$ is $1$ ($M_1 = 1/3$).
Coefficient: $a_5 = \lfloor (1/3) / (1/3) \rfloor = 1$.
Remainder: $R_5 = \frac{1}{3} - 1 \cdot \frac{1}{3} = 0$.

The algorithm is complete. The canonical form of the number 19 is:
$$ 19 = M_5 + M_4 + M_3 + M_2 + M_1. $$
Verification of the canonicity criterion for the highest element $M_5$:
$\sum_{j=1}^4 M_j = 5 + \frac{7}{3} + 1 + \frac{1}{3} = 6 + \frac{8}{3} = \frac{26}{3}$. Since $\frac{26}{3} < \frac{31}{3} (M_5)$, the condition of greedy packing is strictly met.

\section{Equivalence of Spectral and Canonical Forms}\label{app:equivalence}

Using the number $x = 7$ as an example, we demonstrate the elimination of negative coefficients from the spectral form using exclusively the splitting identity \eqref{eq:splitting}, which implies the recursive replacement of the higher digit: $M_{k+1} = 3M_k - 2M_{k-1}$ based on the reverse form of the splitting identity.

\textbf{1. Spectral Form}
For $3x = 21 = 10101_2$, the difference operator yields the spectrum $\vec{c} = [1, -1, 1, -1, 1]$.
The spectral form is:
\begin{equation}
7 = M_5 - M_4 + M_3 - M_2 + M_1.
\end{equation}

\textbf{2. Algebraic Transformation (Descending Cascade)}
We replace the highest element $M_5$ using the identity $M_5 = 3M_4 - 2M_3$:
$$ 7 = (3M_4 - 2M_3) - M_4 + M_3 - M_2 + M_1 = 2M_4 - M_3 - M_2 + M_1. $$
To eliminate the $-M_3$ component, we split one of the $M_4$ elements using the identity $M_4 = 3M_3 - 2M_2$:
$$ 7 = M_4 + (3M_3 - 2M_2) - M_3 - M_2 + M_1 = M_4 + 2M_3 - 3M_2 + M_1. $$
To eliminate $-3M_2$, we split one $M_3$ element using the identity $M_3 = 3M_2 - 2M_1$:
$$ 7 = M_4 + M_3 + (3M_2 - 2M_1) - 3M_2 + M_1 = M_4 + M_3 - M_1. $$
To eliminate the residual deficit $-M_1$, we completely split $M_3$:
$$ 7 = M_4 + (3M_2 - 2M_1) - M_1 = M_4 + 3M_2 - 3M_1. $$

\textbf{3. Resolution of Boundary Overflow}
We have obtained a form with coefficients exceeding the allowed alphabet. We apply the boundary condition of the splitting identity for $k=1$ (where $M_0 = 0$):
$$ 3M_1 = M_2 + 2M_0 \implies 3M_1 = M_2. $$
Substituting $3M_1 = M_2$ into the obtained equation annihilates the deficit:
$$ 7 = M_4 + 3M_2 - M_2 = M_4 + 2M_2. $$

The final form $x = 2M_2 + M_4$ strictly consists of elements from the alphabet $\{0, 1, 2\}$ and satisfies the canonicity condition ($2M_2 = 2 < M_4 = 5$). The transition is mathematically validated by the base operators of the system.
\end{document}